\documentclass[11pt,a4paper,leqno]{amsart}
\usepackage{a4wide}
\usepackage{latexsym}
\usepackage{amsmath}
\usepackage{amssymb}
\usepackage{stackrel} 
\usepackage{color}
\usepackage{graphicx}
\usepackage{hyperref}

\newtheorem{lemma}{Lemma}
\newtheorem{prop}{Proposition}
\newtheorem{theo}{Theorem}
\newtheorem{corol}{Corollary}
\theoremstyle{definition}
\newtheorem{defin}{Definition}

\newtheorem{example}{Example}
\newcommand{\C}{\mathbb{C}}
\newcommand{\N}{\mathbb{N}}
\newcommand{\R}{\mathbb{R}}

\makeatletter
\def\author@andify{%
  \nxandlist {\unskip ,\penalty-1 \space\ignorespaces}%
    {\unskip {} \@@and~}%
    {\unskip \penalty-2 \space \@@and~}%
}
\makeatother

\begin{document}
\title[Summability of formal solutions for some generalized moment PDEs]{Summability of formal solutions for some generalized moment partial differential equations}
\author{Alberto Lastra}
\address{Departamento de F\'isica y Matem\'aticas\\
University of Alcal\'a\\
Ap. de Correos 20, E-28871 Alcal\'a de Henares (Madrid), Spain}
\email{alberto.lastra@uah.es}
\author{S{\l}awomir Michalik}
\address{Faculty of Mathematics and Natural Sciences,
College of Science\\
Cardinal Stefan Wyszy\'nski University\\
W\'oycickiego 1/3,
01-938 Warszawa, Poland}
\email{s.michalik@uksw.edu.pl}
\urladdr{\url{http://www.impan.pl/~slawek}}
\author{Maria Suwi\'nska}
\address{Faculty of Mathematics and Natural Sciences,
College of Science\\
Cardinal Stefan Wyszy\'nski University\\
W\'oycickiego 1/3,
01-938 Warszawa, Poland}
\email{m.suwinska@op.pl}
\date{}
\begin{abstract}
The concept of moment differentiation is extended to the class of moment summable functions, giving rise to moment differential properties. The main result leans on accurate upper estimates for the integral representation of the moment derivatives of functions under exponential-like growth at infinity, and appropriate deformation of the integration paths. The theory is applied to obtain summability results of certain family of generalized linear moment partial differential equations with variable coefficients. 

\end{abstract}

\maketitle
\thispagestyle{empty}

\section{Introduction}

This work is devoted to the study of the summability properties of formal solutions of moment partial differential equations in the complex domain. The purpose of this work is twofold. On the one hand, a deeper knowledge on the moment derivative operator acting on certain functional spaces of analytic functions is put into light; and on the other hand, the previous knowledge serves as a tool  to attain summability results of the formal solutions of concrete families of Cauchy problems. 

The study of moment derivatives, generalizing classical ones, and the solution of moment partial differential equations is of increasing interest in the scientific community. The concept of moment derivative was put forward by W. Balser and M. Yoshino in 2010, in~\cite{BY}. Given a sequence of positive real numbers (in practice a sequence of moments), say $m:=(m(p))_{p\ge0}$, the operator of moment derivative $\partial_{m,z}:\C[[z]]\to\C[[z]]$ acts on the space of formal power series with complex coefficients into itself in the following way (see Definition~\ref{def260}):
$$\partial_{m,z}\left(\sum_{p\ge0}\frac{a_p}{m(p)}z^p\right)=\sum_{p\ge0}\frac{a_{p+1}}{m(p)}z^p.$$
This definition can be naturally extended to holomorphic functions defined on a neighborhood of the origin.

The choice $m=(\Gamma(1+p))_{p\ge0}=(p!)_{p\ge0}$ determines the usual derivative operator, whereas $m=\left(\Gamma\left(1+\frac{p}{s}\right)\right)_{p\ge0}$ is linked to the Caputo $1/s$-fractional differential operator $\partial_{z}^{1/s}$ (see~\cite{M}, Remark 3). Given $q\in (0,1)$ and $m=([p]_{q}!)_{p\ge0}$, with $[p]!_q=[1]_q[2]_q\cdots[p]_{q}$ and $[h]_q=\sum_{j=0}^{h-1}q^{j}$, the operator $\partial_{m,z}$ coincides with the $q$-derivative $D_{q,z}$ defined by
$$D_{q,z}f(z)=\frac{f(qz)-f(z)}{qz-z}.$$
Several recent studies of the previous functional equations have been made in  the complex domain and in terms of summability of their formal solutions, such as~\cite{michalik10} regarding summability of fractional linear partial differential equations;~\cite{immink,immink2} in the study of difference equations; or~\cite{dreyfus,ichinobeadachi,lastramalek} in the study of $q$-difference-differential equations. 

In the more general framework of moment partial differential equations, the seminal work~\cite{BY} was followed by other studies such as~\cite{M} where the second author solves certain families of Cauchy problems under the action of two moment derivatives. We also refer to~\cite{michalik13jmaa,michalik17fe} and~\cite{sanzproceedings} (Section 7), where conditions on the convergence and summability of formal solutions to homogeneous and inhomogeneous linear moment partial differential equations in two complex variables with constant coefficients are stated. Further studies of moment partial differential equations with constant coefficients are described in~\cite{lastramaleksanz}, or in~\cite{michaliktkacz} when dealing with the Stokes phenomenon, closely related to the theory of summability. We also cite~\cite{lastramaleksanz2}, where the moments govern the growth of the elements involved in the problem under study.

A first step towards the study of summability of the formal solution of a functional equation is that of determining the growth rate of its coefficients, which is described in the works mentioned above, and also more specifically in the recent works~\cite{LMS,michaliksuwinska,suwinska} when dealing with moment partial differential equations with constant and time-dependent coefficients. See also the references therein for a further knowledge on the field.

The present work takes a step forward into the theory of generalized summability of formal solutions of moment partial differential equations. The first main result (Theorem~\ref{lemma1}) determines the integral representation of the moment derivatives ($m$-derivatives) of an analytic function defined on an infinite sector with the vertex at the origin together with a neighborhood of the origin, with prescribed exponential-like growth governed by a second sequence, say $\tilde{\mathbb{M}}$. In addition to this, accurate upper estimates of such derivatives are provided showing the same exponential-like growth at infinity, but also its dependence on the moment sequence $m$. This result entails that the set of $\tilde{\mathbb{M}}$-summable formal power series along certain direction $d\in\R$, $\C\{z\}_{\tilde{\mathbb{M}},d}$ (see Definition~\ref{defi271} and Theorem~\ref{teo1}), is closed under the action of the operator $\partial_{m,z}$. As a consequence, it makes sense to extend the definition of $\partial_{m,z}$ to $\C\{z\}_{\tilde{\mathbb{M}},d}$ (Definition~\ref{def487}) and also to provide analogous estimates as above for the $m$-derivatives of the elements in $\C\{z\}_{\tilde{\mathbb{M}},d}$ (Proposition~\ref{prop497}).

We apply the previous theory to achieve summability results on moment partial differential equations of the form
\begin{equation}\label{epralintro}
\left\{ \begin{array}{lcc}
            \left(\partial_{m_1,t}^{k}-a(z)\partial_{m_2,z}^{p}\right)u(t,z)=\hat{f}(t,z)&\\
             \partial_{m_1,t}^{j}u(0,z)=\varphi_j(z),&\quad j=0,\ldots,k-1,
             \end{array}
\right.
\end{equation}
where $1\le k<p$ are integer numbers and $m_1,\,m_2$ are moment sequences under additional assumptions. The elements $a(z),\,a(z)^{-1},\,\varphi_j(z)$ for $j=0,\ldots,k-1$ are assumed to be holomorphic functions in a neighborhood of the origin, and $\hat{f}\in\C[[t,z]]$. The second main result of this research (Theorem~\ref{teopral}) states that summability of the unique formal solution of (\ref{epralintro}) $\hat{u}(t,z)$ (with respect to $z$ variable) along direction $d\in\R$ is equivalent to summability of $\hat{f}$ and $\partial_{m_2,z}^{j}\hat{u}(t,0)$, for $j=0,\ldots,p-1$, along $d$. A result on the convergence of the formal solution is also provided (Corollary~\ref{corofinal}). It is worth mentioning that the results on the upper estimates of formal solutions obtained in~\cite{LMS} remain coherent with these results, and also with those in~\cite{remy2016}, in the Gevrey classical settings. The study of more general moment problems remains open and it is left for future research of the authors.

The paper is structured as follows: After a section describing the notation followed in the present study (Section~\ref{secnot}), we recall the main concepts and results on the generalized moment differentiation of formal power series. Section~\ref{sec31} is devoted to recalling the main tools associated with strongly regular sequences and some of their related elements. In Section~\ref{secfun}, based on the general moment summability methods, we state the first main result of the paper (Theorem~\ref{lemma1}) and its main consequences. The work is concluded in Section~\ref{secapp} with the application of the theory to the summability of formal solutions of certain family of Cauchy problems involving moment partial differential equations in the complex domain (Theorem~\ref{teopral}).

\section{Notation}\label{secnot}

Let $\mathbb{N}$ denote the set of natural numbers
$\{1,2,\cdots\}$ and $\mathbb{N}_0:=\mathbb{N}\cup\{0\}$.

$\mathcal{R}$ stands for the Riemann surface of the logarithm. 

Let $\theta>0$ and $d\in\R$. We write $S_{d}(\theta)$ for the open infinite sector contained in the Riemann surface of the logarithm with the vertex at the origin, bisecting direction $d\in\R$ and opening $\theta>0$, i.e.,
$$S_{d}(\theta):=\left\{z\in\mathcal{R}: |\hbox{arg}(z)-d|<\frac{\theta}{2}\right\}.$$
We write $S_d$ in the case when the opening $\theta>0$ does not need to be specified. A sectorial region $G_d(\theta)$ is a subset of $\mathcal{R}$ such that $G_d(\theta)\subseteq S_{d}(\theta)\cap D(0,r)$ for some $r>0$, and for all $0<\theta'<\theta$ there exists $0<r'<r$ such that $(S_d(\theta')\cap D(0,r'))\subseteq G_d(\theta)$. We denote by $\hbox{arg}(S)$ the set of arguments of $S$, in particular $\hbox{arg}(S_{d}(\theta))=\left(d-\frac{\theta}{2},d+\frac{\theta}{2}\right)$.

We put $\hat{S}_{d}(\theta;r):=S_{d}(\theta)\cup D(0,r)$. Analogously, we write $\hat{S}_{d}(\theta)$ (resp. $\hat{S}_{d}$) whenever the radius $r>0$ (resp. the radius and the opening $r,\,\theta>0$) can be omitted. We write $S\prec S_d(\theta)$ whenever $S$ is an infinite sector with the vertex at the origin with $\overline{S}\subseteq S_d(\theta)$. Analogously, we write $\hat{S}\prec \hat{S}_d(\theta;r)$ if $\hat{S}=S\cup D(0,r')$, with $S\prec S_{d}(\theta)$ and $0<r'<r$. Given two sectorial regions $G_d(\theta)$ and $G_{d'}(\theta')$, we use notation $G_d(\theta)\prec G_{d'}(\theta')$ whenever this relation holds for the sectors involved in the definition of the corresponding sectorial regions.

Given a complex Banach space $(\mathbb{E},\left\|\cdot\right\|_{\mathbb{E}})$, the set $\mathcal{O}(U,\mathbb{E})$ stands for the set of holomorphic functions in a set $U\subseteq\C$, with values in $\mathbb{E}$. If $\mathbb{E}=\C$, then we simply write $\mathcal{O}(U)$. We denote the formal power series with coefficients in $\mathbb{E}$ by $\mathbb{E}[[z]]$.

Given $\hat{f},\hat{g}\in\mathbb{E}[[z]]$, with $\hat{f}(z)=\sum_{p\ge0}f_pz^p$ and $\hat{g}(z)=\sum_{p\ge0}g_pz^p$, such that $g_p\ge0$ for all $p\ge 0$, we write $\hat{f}(z)\ll\hat{g}(z)$ if $|f_p|\le g_p$ for all $p\ge0$.

\section{On generalized summability and moment differentiation}

The aim of this section is to recall the concept and main results on the so-called generalized moment differentiation of formal power series. Certain algebraic properties associated with the families of analytic functions which are related to this notion allow to go further by defining the moment differentiation associated with the sum of a given formal power series.

\subsection{Strongly regular sequences and related elements}\label{sec31}

As a first step, we recall the main tools associated with strongly regular sequences and some of their related elements. The concept of a strongly regular sequence is put forward by V. Thilliez in~\cite{thilliez}.

\begin{defin}
Let $\mathbb{M}:=(M_{p})_{p\ge0}$ be a sequence of positive real numbers with $M_0=1$. 
\begin{itemize}
\item[$(lc)$] The sequence $\mathbb{M}$ is logarithmically convex if 
$$M_p^2\le M_{p-1}M_{p+1},\hbox{ for all }p\ge1.$$
\item[$(mg)$] The sequence $\mathbb{M}$ is of moderate growth if there exists $A_1>0$ such that $$M_{p+q}\le A_1^{p+q}M_{p}M_{q},\hbox{ for all }p,q\ge 0.$$
\item[$(snq)$] The sequence $\mathbb{M}$ satisfies the strong non-quasianalyticity condition if there exists $A_2>0$ such that
$$\sum_{q\ge p}\frac{M_q}{(q+1)M_{q+1}}\le A_2\frac{M_p}{M_{p+1}},\hbox{ for all }p\ge0.$$
\end{itemize}
Any sequence satisfying properties $(lc)$, $(mg)$ and $(snq)$ is known as a \textit{strongly regular sequence}.
\end{defin}

It is worth recalling that given a (lc) sequence $\mathbb{M}=(M_p)_{p\ge0}$, one has 
\begin{equation}\label{e136}
M_pM_q\le M_{p+q}, \hbox{ for all }p,q\in\N_0
\end{equation}
(see Proposition 2.6 (ii.2)~\cite{sanzproceedings}), which entails that given $s\in\N_0$, one has
\begin{equation}\label{e140}
M_p^{s}\le M_{ps},\hbox{ for all }p\in\N_0,
\end{equation}
following an induction argument.

Examples of such sequences are of great importance in the study of formal and analytic solutions of differential equations. Gevrey sequences are predominant among them, appearing as upper bounds for growth of the coefficients of the formal solutions of such equations. Given $\alpha>0$, the Gevrey sequence of order $\alpha$ is defined by $\mathbb{M}_{\alpha}:=(p!^{\alpha})_{p\ge0}$. A natural generalization of the previous are the sequences defined by $\mathbb{M}_{\alpha,\beta}:=(p!^{\alpha}\prod_{m=0}^{p}\log^{\beta}(e+m))_{p\ge0}$ for $\alpha>0$ and $\beta\in\R$. These sequences turn out to be strongly regular sequences, provided that the first terms are slightly modified in the case that $\beta<0$, without affecting their asymptotic behavior. Formal solutions of difference equations are quite related to the 1+ level, associated with the sequence $\mathbb{M}_{1,-1}$, see~\cite{immink,immink2}. 

Given a strongly regular sequence $\mathbb{M}=(M_p)_{p\ge0}$, one can define the function
\begin{equation}\label{e123}
M(t):=\sup_{p\ge0}\log\left(\frac{t^p}{M_p}\right),\quad t>0,\qquad M(0)=0,
\end{equation}
which is non-decreasing, and continuous in $[0,\infty)$ with $\lim_{t\to\infty}M(t)=+\infty$. J. Sanz (\cite{sanz},~Theorem 3.4) proves that the order of the function $M(t)$, defined in~\cite{goldbergostrovskii} by
$$\rho(M):=\limsup_{r\to\infty}\max\left\{0,\frac{\log (M(r))}{\log(r)}\right\}$$
is a positive real number. Moreover, its inverse $\omega(\mathbb{M}):=1/\rho(M)$ determines the limit opening for a sector in such a way
that nontrivial flat function in ultraholomorphic classes of functions defined on such sectors exist, see Corollary 3.16,~\cite{jjs19}. Indeed, $\omega(\mathbb{M})$ can be recovered directly from $\mathbb{M}$ under some admissibility conditions on the sequence $\mathbb{M}$ (Corollary 3.10,~\cite{jjs17}):
$$\lim_{p\to\infty}\frac{\log(M_{p+1}/M_p)}{\log(p)}=\omega(\mathbb{M}).$$ 
Such conditions are satisfied by the sequences of general use in the asymptotic theory of solutions to functional equations.

The next results can be found in~\cite{chaumatchollet,thilliez} under more general assumptions.
\begin{lemma}\label{lemaaux1}
Let $\mathbb{M}$ be a strongly regular sequence, and let $s\ge1$. Then, there exists $\rho(s)\ge1$ which only depends on $\mathbb{M}$ and $s$, such that 
$$ \exp\left(-M(t)\right)\le \exp(-sM(t/\rho(s))),$$
for all $t\ge0$.
\end{lemma}

In view of Lemma 1.3.4~\cite{thilliez}, one has that given a strongly regular sequence, $\mathbb{M}=(M_{p})_{p\ge0}$ and $s>0$. Then, the sequence $\mathbb{M}^s=(M_p^{s})_{p\ge0}$ is strongly regular, and $\omega(\mathbb{M}^s)=\omega(s\mathbb{M})$.

Following~\cite{petzsche}, one has the next definition.

\begin{defin}
Given two sequences of positive real numbers, $\mathbb{M}=(M_p)_{p\ge0}$ and $\tilde{\mathbb{M}}:=(\tilde{M}_p)_{p\ge 0}$, we say that $\mathbb{M}$ and $\tilde{\mathbb{M}}$ are \textit{equivalent} if there exist $B_1,B_2>0$ with
\begin{equation}\label{e195}
B_1^p M_p\le \tilde{M}_p\le B_2^p M_p,
\end{equation}
for every $p\ge0$.
\end{defin}

The next result is a direct consequence of the definition of the function $M$ in (\ref{e123}). 

\begin{lemma}\label{lemma2}
Let $\mathbb{M},\,\tilde{\mathbb{M}}$ be two strongly regular sequences which are equivalent. Let $M(t)$ (resp. $\tilde{M}(t)$) be the function associated with $\mathbb{M}$ (resp. $\tilde{\mathbb{M}}$) through (\ref{e123}). Then, it holds that
$$ M\left(\frac{t}{B_2}\right)\le \tilde{M}(t)\le M\left(\frac{t}{B_1}\right)  \hbox{ for all }t\ge0,$$
where $B_1,\,B_2$ are the positive constants in (\ref{e195}). 
\end{lemma}

\begin{defin}
Let $(M_p)_{p\ge0}$ be a sequence of positive real numbers with $M_0=1$, and let $s\in\R$. A sequence of positive real numbers $(m(p))_{p\ge0}$ is said to be an $(M_p)$-sequence of order $s$ if there exist $A_3,A_4>0$ with 
\begin{equation}\label{e100}
A_3^p(M_p)^{s}\le m(p)\le A_4^p(M_p)^s,\quad p\ge0.
\end{equation}
\end{defin}

\subsection{Function spaces and generalized summability}\label{secfun}

In the whole subsection $(\mathbb{E},\left\|\cdot\right\|_{\mathbb{E}})$ stands for a complex Banach space.

\begin{defin}
Let $r,\,\theta>0$ and $d\in\R$. We also fix a sequence $\mathbb{M}$ of positive real numbers. The set $\mathcal{O}^{\mathbb{M}}(\hat{S}_d(\theta;r),\mathbb{E})$ consists of all functions $f\in\mathcal{O}(\hat{S}_d(\theta;r),\mathbb{E})$ such that for every $0<\theta'<\theta$ and $0<r'<r$ there exist $\tilde{c},\tilde{k}>0$ with
\begin{equation}\label{e144}
\left\|f(z)\right\|_{\mathbb{E}}\le \tilde{c}\exp\left(M\left(\frac{|z|}{\tilde{k}}\right)\right),\quad z\in \hat{S}_d(\theta',r').
\end{equation}
Analogously, the set $\mathcal{O}^{\mathbb{M}}(S_d(\theta),\mathbb{E})$ consists of all $f\in\mathcal{O}(S_d(\theta),\mathbb{E})$ such that for all $0<\theta'<\theta$, there exist $\tilde{c},\tilde{k}>0$ such that (\ref{e144}) holds for all $z\in S_d(\theta')$.
\end{defin}

The aforementioned definition generalizes that of functions of exponential growth at infinity of some positive order. Indeed, if $\mathbb{M}=\mathbb{M}_{\alpha}$ for some $\alpha>0$, then the property (\ref{e144}) determines that $f$ is of exponential growth at most $1/\alpha$.

The general moment summability methods developed by W. Balser  (see Section 5.5,~\cite{balser}) were adapted by J. Sanz to the framework of strongly regular sequences (see Section 5,~\cite{sanz}; or Definition 6.2.,~\cite{sanzproceedings}).

\begin{defin}
Let $\mathbb{M}$ be a strongly regular sequence with $\omega(\mathbb{M})<2$. Let $M$ be the function associated with $\mathbb{M}$, defined by (\ref{e123}). The complex functions $e,\,E$ define \textit{kernel functions for $\mathbb{M}$-summability} if the following properties hold:
\begin{enumerate}
\item $e\in\mathcal{O}(S_0(\omega(\mathbb{M})\pi))$. The function $e(z)/z$ is locally uniformly integrable at the origin, i.e., there exists $t_0>0$, and for all $z_0\in S_{0}(\omega(\mathbb{M})\pi)$ there exists a neighborhood $U$ of $z_0$, with $U\subseteq S_0(\omega(\mathbb{M})\pi)$, such that 
\begin{equation}\label{e192}
\int_0^{t_0}\frac{\sup_{z\in U}\left|e\left(\frac{t}{z}\right)\right|}{t}dt<\infty.
\end{equation}
Moreover, for all $\epsilon>0$ there exist $c, k>0$ such that
\begin{equation}\label{e162}
|e(z)|\le c\exp\left(-M\left(\frac{|z|}{k}\right)\right)\quad \hbox{ for all }z\in S_0(\omega(\mathbb{M})\pi-\epsilon).
\end{equation}
We also assume that $e(x)\in\R$ for all $x>0$.
\item $E\in\mathcal{O}(\C)$ and satisfies that 
\begin{equation}\label{e191}
|E(z)|\le \tilde{c}\exp\left(M\left(\frac{|z|}{\tilde{k}}\right)\right),\quad z\in \C,
\end{equation}
for some $\tilde{c},\,\tilde{k}>0$. There exists $\beta>0$ such that for all $0<\tilde{\theta}<2\pi-\omega(\mathbb{M})\pi$ and $M_E>0$, there exist $\tilde{c}_2>0$ with 
\begin{equation}\label{e202}
|E(z)|\le \frac{\tilde{c}_2}{|z|^{\beta}},\quad z\in S_\pi(\tilde{\theta})\setminus D(0,M_E).
\end{equation}
\item Both kernel functions are related via the Mellin transform of $e$. More precisely, the \textit{moment function} associated with $e$, defined by
\begin{equation}\label{e166}
m_{e}(z):=\int_{0}^{\infty} t^{z-1}e(t)dt
\end{equation}
is a complex continuous function in $\{z\in\C:\hbox{Re}(z)\ge 0\}$ and holomorphic in $\{z\in\C:\hbox{Re}(z)> 0\}$. The kernel function $E$ has the power series expansion at the origin given by
\begin{equation}\label{e167}
E(z)=\sum_{p\ge0}\frac{z^p}{m_{e}(p)},\quad z\in\C.
\end{equation}
\end{enumerate}
\end{defin}

\textbf{Remark:} In the remainder of the work we will only mention the kernel function $e$, rather than the pair $e,\,E$, as $E$ is determined by the knowledge of $e$ in terms of its Taylor expansion at the origin.

\vspace{0.3cm}

\textbf{Remark:} The growth condition of the kernel function $E(z)$ at infinity (\ref{e202}) is usually substituted in the literature (\cite{balser,michalik17fe,sanz,sanzproceedings}) by the less restrictive condition:

\textit{``The function $E(1/z)/z$ is locally uniformly integrable at the origin in the sector $S_{\pi}((2-\omega(\mathbb{M}))\pi)$. Namely, there exists $t_0>0$, and for all $z_0\in S_{\pi}((2-\omega(\mathbb{M}))\pi)$ there exists a neighborhood $U$ of $z_0$, $U\subseteq S_{\pi}((2-\omega(\mathbb{M}))\pi)$, such that
$$\int_0^{t_0}\frac{\sup_{z\in U}\left|E\left(\frac{z}{t}\right)\right|}{t}dt<\infty.\hbox{''}$$}

Condition (\ref{e202}) has already been used and justified in~\cite{jkls} (see Lemma 4.10, Remark 4.11 and Remark 4.12 in~\cite{jkls}), in order to obtain convolution kernels for multisummability.

\begin{defin}
Let $\mathbb{M}$ be a strongly regular sequence and let $e,\,E$ be a pair of kernel functions for $\mathbb{M}$-summability. Let $m_e$ be the moment function given by (\ref{e166}). The sequence $(m_{e}(p))_{p\ge0}$ is the so-called \textit{sequence of moments} associated with $e$.
\end{defin}

\textbf{Remark:} 
The previous definition can be adapted to the case $\omega(\mathbb{M})\ge2$ by means of a ramification of the kernels (see~\cite{sanzproceedings}, Remark 6.3 (iii)). For practical purposes, we will focus on the case that $\omega(\mathbb{M})<2$, taking into consideration that all the results can be adapted to the general case.

\vspace{0.3cm}

\textbf{Remark:}\label{r223}
Given a strongly regular sequence $\mathbb{M}$, the existence of pairs of kernel functions for $\mathbb{M}$-summability is guaranteed, provided that $\mathbb{M}$ admits a nonzero proximate order (see~\cite{jjs17,lastramaleksanz}). 

\vspace{0.3cm}

\begin{example}\label{ex187}
Let $\alpha>0$. We consider a Gevrey sequence $\mathbb{M}_{\alpha}$. Then, the functions $e_{\alpha}(z):=\frac{1}{\alpha}z^{\frac{1}{\alpha}}\exp\left(-z^{\frac{1}{\alpha}}\right)$ and $E_{\alpha}(z):=\sum_{p\ge0}\frac{z^p}{\Gamma(1+\alpha p)}$ are kernel functions for $\mathbb{M}_{\alpha}$-summability. Indeed, the moment function is given by $m_{\alpha}(z):=\Gamma(1+\alpha z)$.
\end{example}

The definition of moment differentiation, moment (formal) Borel and moment Laplace transformation generalize the classical concepts of differentiation, formal Borel and Laplace transformations, respectively. In the classical setting of the Gevrey sequence of order $\alpha>0$, the moment sequence is $(\Gamma(1+\alpha p))_{p\ge0}$ seen in Example~\ref{ex187}. Classical differentiation corresponds to $\alpha=1$. 

At this point, given a strongly regular sequence $\mathbb{M}$, one has that a sequence of moments can be constructed, provided a pair of kernel functions for $\mathbb{M}$-summability, say $e$ and $E$. The associated sequence of moments $m_e:=(m_e(p))_{p\ge0}$ is a strongly regular sequence (see~\cite{sanzproceedings}, Remark 6.6), which is equivalent to $\mathbb{M}$ (see~\cite{sanzproceedings}, Proposition 6.5). Therefore, $\omega(\mathbb{M})=\omega(m_{e})$. The definition of generalized derivatives is done in terms of a sequence of moments, rather than the initial sequence itself, and we will work directly  with this sequence, obviating the initial strongly regular sequence and the pair of kernel functions. Hereinafter, when referring to a sequence of moments, we will assume without mentioning that such sequence is indeed the sequence of moments associated with some pair of kernel functions, and therefore with a strongly regular sequence (in conditions that admit such pair of kernels, e.g., if the strongly regular sequence admits a nonzero proximate order).

Departing from a sequence of moments $m_e$, one can consider the formal $m_e$-moment Borel transform. This definition can be extended for any sequence of positive numbers, and not only to a sequence of moments. We present it in this way for the sake of clarity.

\begin{defin}\label{def259}
Let $m_e=(m_e(p))_{p\ge0}$ be a sequence of moments. The formal $m_e$-moment Borel transform $\hat{\mathcal{B}}_{m_e,z}:\mathbb{E}[[z]]\to\mathbb{E}[[z]]$ is defined by
$$\hat{\mathcal{B}}_{m_e,z}\left(\sum_{p\ge0}a_pz^p\right)=\sum_{p\ge0}\frac{a_p}{m_e(p)}z^p.$$
\end{defin}

There are several different equivalent approaches to the general moment summability of formal power series. We refer to~\cite{balser}, Section 6.5, under Gevrey-like settings, and~\cite{sanzproceedings}, Section 6, in the framework of strongly regular sequences.  

\begin{defin}\label{defi271}
Let $\mathbb{M}$ be a strongly regular sequence admitting a nonzero proximate order. The formal power series $\hat{u}\in\mathbb{E}[[z]]$ is $\mathbb{M}$-summable in direction $d\in\R$ if the formal power series $\hat{\mathcal{B}}_{\mathbb{M},z}(\hat{u}(z))$ is convergent in a neighborhood of the origin and can be extended to an infinite sector of bisecting direction $d$, say $\hat{S}_d$, such that the extension belongs to $\mathcal{O}^{\mathbb{M}}(\hat{S}_d,\mathbb{E})$. We write $\mathbb{E}\{z\}_{\mathbb{M},d}$ for the set of $m_e$-summable formal power series in $\mathbb{E}[[z]]$. Here we have assumed that $e$ is a kernel function for $\mathbb{M}$-summability and $m_e$ is its associated sequence of moments.   
\end{defin}

\textbf{Remark:} We recall that, given a sequence of moments associated with a strongly regular sequence $\mathbb{M}$ via a kernel function $e$, say $m_e$, $\mathbb{M}$ and $m_e$ are equivalent sequences. Regarding Lemma~\ref{lemma2} and the definition of the formal Borel transformation, it is easy to check that the set $\mathbb{E}\{z\}_{m_e,d}$ does not depend on the choice of the kernel function $e$, and therefore one can write $\mathbb{E}\{z\}_{\mathbb{M},d}:=\mathbb{E}\{z\}_{m_e,d}$ for any choice of kernel function $e$. Moreover, we observe that the formal power series $\hat{\mathcal{B}}_{m_e,z}(\hat{u})$ has a positive radius of convergence with independence of the kernel function considered, associated with $\mathbb{M}$.

\vspace{0.3cm}

The statements in the next proposition can be found in detail in~\cite{sanzproceedings}, Section 6.

\begin{prop}\label{prop316}
Let $d\in\R$ and let $e,\,E$ be a pair of kernel functions for $\mathbb{M}$-summability. Let $\theta>0$. For every $f\in\mathcal{O}^{\mathbb{M}}(S_d(\theta),\mathbb{E})$, the $e$-\emph{Laplace transform} of $f$ along a direction $\tau\in\hbox{arg}(S_d(\theta))$ is defined by
$$(T_{e,\tau}f)(z)=\int_0^{\infty(\tau)}e(u/z)f(u)\frac{du}{u},$$
for $|\hbox{arg}(z)-\tau|<\omega(\mathbb{M})\pi/2$, and $|z|$ small enough. The variation of $\tau\in\hbox{arg}(S_d)$ defines a function denoted by $T_{e,d}f$ in a sectorial region $G_d(\theta+\omega(\mathbb{M})\pi)$.

Under the assumption that $\omega(\mathbb{M})<2$ let $G=G_d(\theta)$ be a sectorial region with $\theta>\omega(\mathbb{M})\pi$. Given $f\in\mathcal{O}(G,\mathbb{E})$ and continuous at 0, and $\tau\in\R$ with $|\tau-d|<(\theta-\omega(\mathbb{M})\pi)/2$, the operator $T^{-}_{e,\tau}$, known as the $e$-\emph{Borel transform} along direction $\tau$ is defined by
$$(T^{-}_{e,\tau}f)(u):=\frac{-1}{2\pi i}\int_{\delta_{\omega(\mathbb{M})}(\tau)}E(u/z)f(z)\frac{dz}{z},\quad u\in S_{\tau},$$
where $S_{\tau}$ is an infinite sector of bisecting direction $\tau$ and small enough opening, and $\delta_{\omega(\mathbb{M})}(\tau)$ is the Borel-like path consisting of the concatenation of a segment from the origin to a point $z_0$ with $\hbox{arg}(z_0)=\tau+\omega(\mathbb{M})(\pi+\epsilon)/2$, for some small enough $\epsilon\in(0,\pi)$, followed with the arc of circle centered at 0, joining $z_0$ and the point $z_1$, with $\hbox{arg}(z_1)=\tau-\omega(\mathbb{M})(\pi+\epsilon)/2$, clockwise, and concluding with the segment of endpoints  $z_1$ and the origin.

Let $G_{d}(\theta)$ and $f$ be as above. The family $\{T^{-}_{e,\tau}\}_{\tau}$, with $\tau$ varying among the real numbers with $|\tau-d|<(\theta-\omega(\mathbb{M})\pi)/2$ defines a holomorphic function denoted by $T^{-}_{e,d}f$ in the sector $S_{d}(\theta-\omega(\mathbb{M})\pi)$ and $T^{-}_{e,d}f\in\mathcal{O}^{\mathbb{M}}(S_{d}(\theta-\omega(\mathbb{M})\pi),\mathbb{E})$.
\end{prop}

\textbf{Remark:} We recall that if $\lambda\in\C$, with $\hbox{Re}(\lambda)\ge0$, then $T^{-}_{e,d}(u\mapsto u^{\lambda})(z)=\frac{z^{\lambda}}{m_e(\lambda)}$, which relates $T^{-}_{e,d}$ with the formal Borel operator $\hat{\mathcal{B}}_{m_e,u}$.

Theorem 30~\cite{balser} can be adapted to the strongly regular sequence framework, under minor modifications.

\begin{theo}\label{teo324}
Let $S=S_d(\theta)$ for some $\theta>0$. Let $\mathbb{M}$ be a strongly regular sequence with $\omega(\mathbb{M})<2$ admitting a nonzero proximate order, and choose a kernel function for $\mathbb{M}$-summability $e$. Let $f\in\mathcal{O}^{\mathbb{M}}(S,\mathbb{E})$ and define $g(z)=(T_{e,d}f)(z)$ for $z$ in a sectorial region $G_{d}(\theta+\omega(\mathbb{M})\pi))$. Then it holds that $f\equiv T^{-}_{e,d}g$.
\end{theo}

The following is an equivalent of $\mathbb{M}$-summable formal power series (see Theorem 6.18,~\cite{sanzproceedings}).

\begin{theo}\label{teo1}
Let $\mathbb{M}=(M_p)_{p\ge0}$ be a strongly regular sequence admitting a nonzero proximate order. Let $\hat{u}=\sum_{p\ge0}u_pz^p\in\mathbb{E}[[z]]$ and $d\in\R$. The following statements are equivalent:
\begin{itemize}
\item[(a)] $\hat{u}(z)$ is $\mathbb{M}$-summable in direction $d$.
\item[(b)] There exists a sectorial region $G_d(\theta)$ with $\theta>\omega(\mathbb{M})\pi$ and $u\in\mathcal{O}(G_d(\theta),\mathbb{E})$ such that for all $0<\theta'<\theta$, $S_d(\theta';r)\subseteq G_d(\theta)$ and all integers $N\ge1$, there exist $C,\,A>0$ with
$$\left\|u(z)-\sum_{p=0}^{N-1}u_{p}z^p\right\|_{\mathbb{E}}\le C A^{N}M_N|z|^{N},\qquad z\in S_d(\theta';r).$$
\end{itemize}
If one of the previous equivalent statements holds, the function $u$ in Definition~\ref{defi271} can be constructed as the $e$-Laplace transform of $\hat{\mathcal{B}}_{\mathbb{M},z}(\hat{u}(z))$ along direction $\tau\in\hbox{arg}(S_d)$.

The previous construction can be done for, and it is independent of, any choice of the kernel for $\mathbb{M}$-summability $e$.
\end{theo}

The function $u$  satisfying the previous equivalent properties is unique (see Corollary 4.30~\cite{sanzproceedings}), and it is known as the \textit{$\mathbb{M}$-sum} of $\hat{u}$ along direction $d$. We write $\mathcal{S}_{\mathbb{M},d}(\hat{u})$ for the $\mathbb{M}$-sum of $\hat{u}$ along direction $d$.

The concept of a moment differential operator was put forward by W. Balser and M. Yoshino in~\cite{BY}, and extended to $m_e$-moment differential operators in~\cite{LMS}, which leans on moment sequences of some positive order. 

\begin{defin}\label{def260}
Let $(\mathbb{E},\left\|\cdot\right\|_{\mathbb{E}})$ be a complex Banach space. Given a sequence of moments $(m_{e}(p))_{p\ge0}$, the $m_e$-moment differentiation $\partial_{m_e,z}$ is the linear operator $\partial_{m_e,z}:\mathbb{E}[[z]]\to\mathbb{E}[[z]]$ defined by
$$\partial_{m_e,z}\left(\sum_{p\ge0}\frac{u_{p}}{m_e(p)}z^p\right):=\sum_{p\ge0}\frac{u_{p+1}}{m_e(p)}z^p.$$
\end{defin}

This definition can be naturally extended to $f\in\mathcal{O}(D,\mathbb{E})$, for some complex Banach space $(\mathbb{E},\left\|\cdot\right\|_{\mathbb{E}})$, and any neighborhood of the origin $D$, by applying the definition of $
\partial_{m_e,z}$ to the Taylor expansion of $f$ at the origin. Moreover, one defines the linear operator $\partial_{m_e,z}^{-1}:\mathbb{E}[[z]]\to\mathbb{E}[[z]]$ as the inverse operator of $\partial_{m_e,z}$, i.e. $\partial_{m_e,z}^{-1}(z^j)=\frac{m_e(j)}{m_e(j+1)}z^{j+1}$ for every $j\ge0$.

\begin{lemma}\label{lema268}
Let $m_1=(m_1(p))_{p\ge0},\,m_2=(m_2(p))_{p\ge0}$ be two sequences of moments. The following statements hold:
\begin{itemize}
\item The sequence of products $m_1m_2:=(m_1(p)m_2(p))_{p\ge0}$ is a sequence of moments.
\item $\hat{\mathcal{B}}_{m_1,z}\circ \partial_{m_2,z}\equiv \partial_{m_1m_2,z}\circ\hat{\mathcal{B}}_{m_1,z}$ as operators defined in $\mathbb{E}[[z]]$.
\end{itemize}
\end{lemma}
\begin{proof}
The first part is a direct consequence of Proposition 4.15,~\cite{jkls}. The second part is a direct consequence of the definition of the formal Borel transform (Definition~\ref{def259}) and the moment differentiation (Definition~\ref{def260}).
\end{proof}


The first statement of the next result extends Proposition 6~\cite{michalik17fe} to the framework of strongly regular sequences. Its proof rests heavily on that of Proposition 3,~\cite{M}. The second statement will be crucial in the sequel at the time of giving a coherent meaning to the moment derivatives of the sum of a formal power series.

\begin{theo}\label{lemma1}
Let $m_e=(m_e(p))_{p\ge0}$ be a sequence of moments, and let $\tilde{\mathbb{M}}$ be a strongly regular sequence. We also fix $d,\,\theta,\,r\in\R$, with $\theta,\,r>0$, and $\varphi\in\mathcal{O}^{\tilde{\mathbb{M}}}(\hat{S}_{d}(\theta;r),\mathbb{E})$. Then there exists $0<\tilde{r}<r$ such that for all $0<\theta_1<\theta$, all $z\in \hat{S}_d(\theta_1;\tilde{r})$ and $n\in\N_0$, the following statements hold:
\begin{enumerate}
\item
\begin{equation}\label{e244}
\partial_{m_e,z}^{n}\varphi(z)=\frac{1}{2\pi i}\oint_{\Gamma_z}\varphi(w)\int_0^{\infty(\tau)}\xi^n E(z\xi)\frac{e(w\xi)}{w\xi}d\xi dw,
\end{equation}
with $\tau=\tau(\omega)\in (-\arg(w)-\frac{\omega(m_e)\pi}{2},-\arg(w)+\frac{\omega(m_e)\pi}{2})$, which depends on $w$. The integration path $\Gamma_z$ is a deformation of the circle $\{|w|=r_1\}$, for any choice of  $0<r_1<r$, which depends on $z$.
\item There exist constants $C_1,C_2,C_3>0$ such that
\begin{equation}\label{e279}
\left\|\partial_{m_e,z}^{n}\varphi(z)\right\|_{\mathbb{E}}\le C_1 C_2^n m_e(n)\exp\left(\tilde{M}(C_3|z|)\right),
\end{equation}
for all $n\in\N_0$ and $z\in \hat{S}_d(\theta_1;\tilde{r})$.
\end{enumerate}
\end{theo}
\begin{proof}
We first give a proof for the first statement. From Taylor expansion of $\varphi$ at the origin and the definition of $m_e$-moment derivatives one has
$$\varphi(z)=m_{e}(0)\sum_{p\ge0}\frac{\partial_{m_e,z}^{p}\varphi(0)}{m_e(p)}z^p,$$
for all $z\in D(0,r)$. The application of the Cauchy integral formula for the derivatives yields
$$\partial_{m_e,z}^{p}\varphi(0)=\frac{m_e(p)}{p!m_e(0)}\varphi^{(p)}(0)=\frac{m_e(p)}{2\pi i m_e(0)}\oint_{|w|=r_1}\frac{\varphi(w)}{w^{p+1}}dw,$$
for any $0<r_1<r$.
Let $w\in\C$ with $|w|=r_1$. Following (\ref{e166}) and the change of variables $x=\xi w$ we derive
$$\frac{m_e(p)}{w^{p+1}}=\int_0^{\infty}x^{p-1}\frac{e(x)}{w^{p+1}}dx=\int_0^{\infty(\tau)}\xi^{p}\frac{e(\xi w)}{\xi w}d\xi,$$
where $\tau=-\arg(w)$. We observe from (\ref{e162}) that the previous equality can be extended to any direction of integration $\tau\in\left(-\arg(w)-\frac{\omega(m_e)\pi}{2},-\arg(w)+\frac{\omega(m_e)\pi}{2}\right)$. Therefore, regarding the definition of the kernel function $E(z)$ in (\ref{e167}), one has
\begin{multline}\varphi(z)=m_e(0)\sum_{p\ge0}\frac{\partial_{m_e,z}^{p}\varphi(0)}{m_e(p)}z^{p}=\frac{1}{2\pi i}\oint_{|w|=r_1}\varphi(w)\int_0^{\infty(\tau)}\frac{e(\xi w)}{\xi w}\sum_{p\ge0}\frac{\xi^pz^{p}}{m_e(p)}d\xi dw \\
=\frac{1}{2\pi i}\oint_{|w|=r_1}\varphi(w)\int_0^{\infty(\tau)}E(\xi z)\frac{e(\xi w)}{\xi w}d\xi dw.\label{e309}
\end{multline}
We conclude the first part of the proof, at least from the formal point of view, by observing that
$$\partial_{m_e,z}^{n}E(\xi z)=\partial_{m_e,z}^{n}\left(\frac{(\xi z)^{p}}{m_e(p)}\right)=\sum_{p\ge0}\frac{\xi^{p+n}z^{p}}{m_e(p)}=\xi^{n}E(\xi z),$$
for every $\xi,z\in\C$. It only rests to guarantee that the formal interchange of sum and integrals in (\ref{e309}) can also be made with analytic meaning. We give details about this issue in the second part of the proof.

We proceed to give proof for the second statement of the result. Let us first consider the integral
$$\int_0^{\infty(\tau)}\xi^nE(z\xi)\frac{e(w\xi)}{w\xi}d\xi,$$
for $z$ belonging to some neighborhood of the origin, $w\in\C$ with $|w|=r_1$ as above. We choose $\tau\in\left(-\arg(w)-\frac{\omega(m_e)\pi}{2},-\arg(w)+\frac{\omega(m_e)\pi}{2}\right)$. We first prove that 
\begin{equation}\label{e327}
\left|\int_0^{\infty(\tau)}\xi^nE(z\xi)\frac{e(w\xi)}{w\xi}d\xi\right|\le A_0 B_0^n m_e(n),
\end{equation}
for some $A_0,\,B_0>0$ and all $n\ge0$. This can be done following analogous arguments as those in the proof of Lemma 7.2,~\cite{sanzproceedings}. Let us consider the parametrization $[0,\infty)\ni s\mapsto se^{i\tau}$. In view of (\ref{e162}) and (\ref{e191}), we have
\begin{multline}
\left|\int_0^{\infty(\tau)}\xi^nE(z\xi)\frac{e(w\xi)}{w\xi}d\xi\right|\le \int_0^{\infty}s^{n}|E(se^{i\tau}z)|\frac{|e(se^{i\tau}w)|}{sr_1}ds\\
\le \frac{c\tilde{c}}{r_1}\int_0^{\infty}s^{n-1}\exp\left(M\left(\frac{s|z|}{\tilde{k}}\right)\right)\exp\left(-M\left(\frac{sr_1}{k}\right)\right)ds,
\end{multline}
for some $c,\,\tilde{c},\,k,\,\tilde{k}>0$. We apply Lemma~\ref{lemaaux1} to arrive at 
\begin{equation}\label{e329}
\exp\left(M\left(\frac{s|z|}{\tilde{k}}\right)\right)\exp\left(-M\left(\frac{sr_1}{k}\right)\right)\le \exp\left(M\left(\frac{s|z|}{\tilde{k}}\right)-2M\left(\frac{sr_1}{\rho(2)k}\right)\right).
\end{equation}
We recall that the function $M$ is a monotone increasing function. Therefore, if $|z|\le \tilde{r}:=\frac{r_1 \tilde{k}}{k\rho(2)}$, then (\ref{e329}) is bounded from above by $\exp(-M(sr_1/(\rho(2)k)))$. Let us write
\begin{equation}\label{e352}
\int_{0}^{\infty}s^{n-1}e^{-M\left(\frac{sr_1}{\rho(2)k}\right)}ds= \int_0^1s^{n-1}e^{-M\left(\frac{sr_1}{\rho(2)k}\right)}ds+\int_1^\infty s^{n-1}e^{-M\left(\frac{sr_1}{\rho(2)k}\right)}ds=I_1+I_2.
\end{equation}
The definition of $M$ guarantees upper bounds for $|I_1|$ which do not depend on $n$. We proceed to study $I_2$.
Bearing in mind the definition of $M$, we arrive at
$$\int_1^\infty s^{n-1}e^{-M\left(\frac{sr_1}{\rho(2)k}\right)}ds\le \left(\frac{\rho(2)k}{r_1}\right)^{n+2}m_e(n+2)\int_1^\infty \frac{1}{s^{3}}ds.$$
The application of $(mg)$ condition on $m_e(n+2)\le A_1^{n+2}m_{e}(2)m_e(n)$ allows to conclude (\ref{e327}) for $z\in D(0,\tilde{r})$. The estimate (\ref{e279}) is attained by applying (\ref{e327}) to (\ref{e244}). More precisely, we have 
\begin{equation}\label{e358}
\left\|\partial_{m_e,z}^{n}\varphi(z)\right\|_{\mathbb{E}}\le \left(\sup_{|w|=r_1}\left\|\varphi(w)\right\|\right)A_0B_0^nm_{e}(n),
\end{equation}
which entails (\ref{e279}) for $z\in D(0,\tilde{r})$.

Let $0<\theta_1<\theta$, and $z\in S_d(\theta_1)$ with $|z|\ge \tilde{r}$. We study (\ref{e279}) in such a domain. We deform the integration path $\{|w|=r_1\}$ as follows. Let $\theta_1<\theta_2<\theta$ and let $R=R(z)=\frac{\rho(2)k}{\tilde{k}}|z|$. We write $\Gamma_z=\Gamma(R)=\Gamma_1+\Gamma_2(R)+\Gamma_3(R)+\Gamma_4(R)$, where $\Gamma_1$ is the arc of the circle joining the points $r_1e^{i(d+\theta_2)}$ and $r_1e^{i(d-\theta_2)}$ counterclockwise, $\Gamma_2(R)$ is the segment $[r_1,R]e^{i(d-\theta_2)}$, $\Gamma_3(R)$ is the arc of the circle joining  the points $Re^{i(d-\theta_2/2)}$ and $Re^{i(d+\theta_2/2)}$ counterclockwise and $\Gamma_4(R)$ is the segment $[r_1,R]e^{i(d+\theta_2/2)}$. Figure~\ref{fig1} illustrates this deformation path.

\begin{figure}
	\centering
		\includegraphics[width=0.5\textwidth]{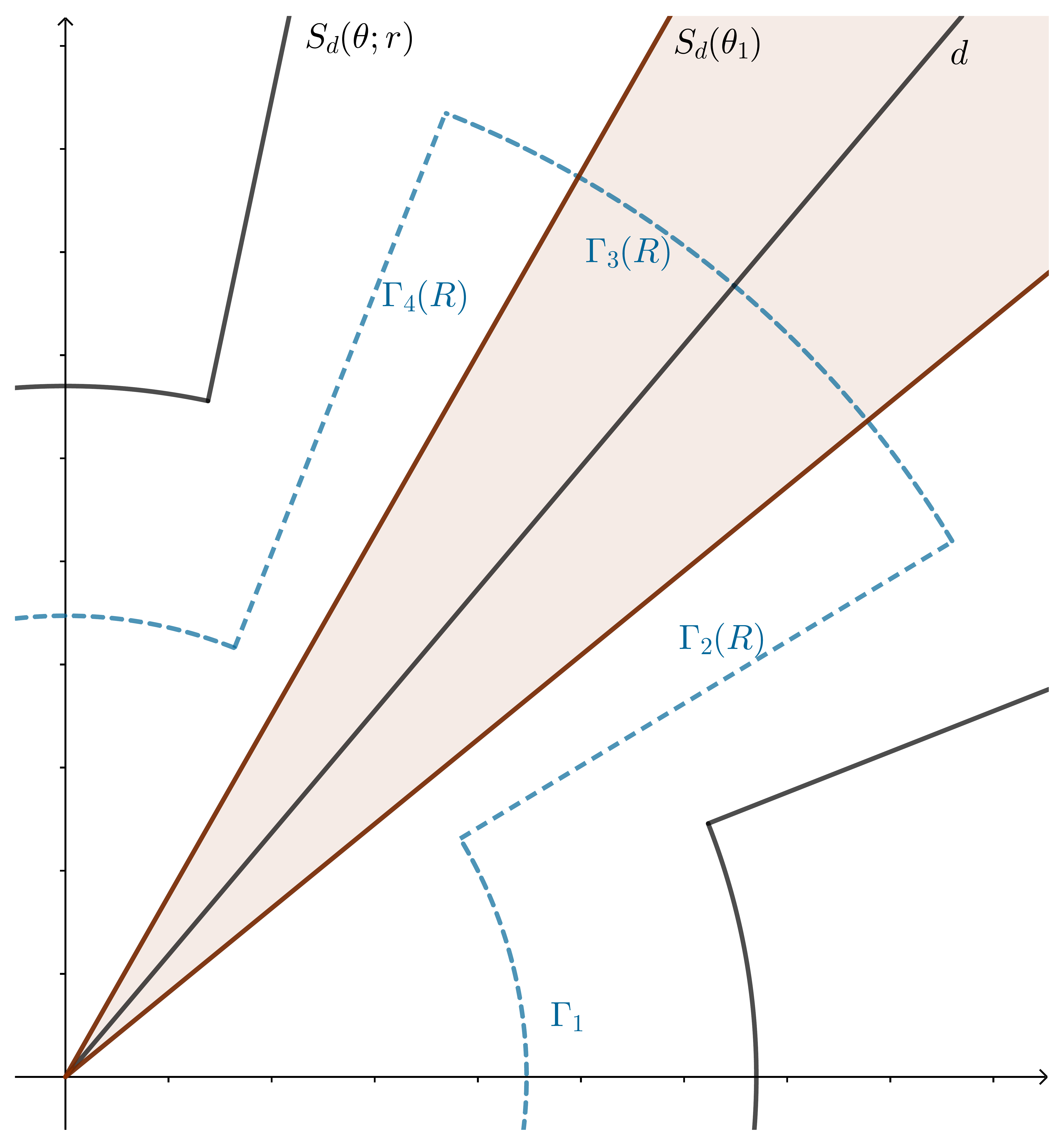}
		\caption{Deformation path}
		\label{fig1}
\end{figure}

We first study the case $\omega\in\Gamma_1$, i.e., $|w|=r_1$. Observe that for every $w\in\Gamma_1$ it is always possible to choose $\tau$ such that
\begin{equation}\label{e379}
\tau\in\left(-\hbox{arg}(w)-\frac{\omega(m_e)\pi}{2},-\hbox{arg}(w)+\frac{\omega(m_e)\pi}{2}\right)\cap \left(-\hbox{arg}(z)+\frac{\omega(m_e)\pi}{2},-\hbox{arg}(z)+2\pi-\frac{\omega(m_e)\pi}{2}\right).
\end{equation}
For one of such directions $\tau$ we have
$$\int_0^{\infty(\tau)}\xi^n E(z\xi)\frac{e(w\xi)}{w\xi}d\xi =\int_0^{\infty}(se^{i\tau})^n E(zse^{i\tau})\frac{e(wse^{i\tau})}{wse^{i\tau}}e^{i\tau}ds.$$
We split the previous integral into two parts. Let $M_E>0$. On the one hand, one has
\begin{multline}
\left|\int_{0}^{M_E/|z|}(se^{i\tau})^n E(zse^{i\tau})\frac{e(wse^{i\tau})}{wse^{i\tau}}e^{i\tau}ds\right|\le \int_{0}^{M_E/|z|}s^n |E(zse^{i\tau})|\frac{|e(wse^{i\tau})|}{r_1s}ds\\
\le \left(\max_{y\in\overline{D}(0,M_E)}|E(y)|\right)\frac{1}{r_1}\int_{0}^{M_E/|z|}s^{n-1} |e(wse^{i\tau})|ds.
\end{multline}
Bearing in mind (\ref{e162}), we have
$$\int_{0}^{M_E/|z|}s^{n-1} |e(wse^{i\tau})|ds\le c\int_{0}^{M_E/|z|}s^{n-1} \exp\left(-M\left(\frac{r_1 s}{k}\right)\right)ds.$$
Analogous estimates as in (\ref{e352}) allow us to arrive at 
\begin{equation}\label{e383}
\left|\int_{0}^{M_E/|z|}(se^{i\tau})^n E(zse^{i\tau})\frac{e(wse^{i\tau})}{wse^{i\tau}}e^{i\tau}ds\right|\le C_{1.1}C_{2.1}^{n}m_e(n)\left(\max_{y\in\overline{D}(0,M_E)}|E(y)|\right)\frac{1}{r_1^{n+3}},
\end{equation}
for some $C_{1.1},\,C_{2.1}>0$. Analogously, we estimate the second integral by means of the upper bounds in (\ref{e202}). Indeed, one has
\begin{align*}
\left|\int_{M_E/|z|}^{\infty}(se^{i\tau})^n E(zse^{i\tau})\frac{e(wse^{i\tau})}{wse^{i\tau}}e^{i\tau}ds\right|&\le \int_{M_E/|z|}^{\infty}s^n \frac{\tilde{c_2}}{(|z|s)^{\beta}}\frac{|e(wse^{i\tau})|}{r_1s}ds\\
&\le\frac{\tilde{c_2}}{r_1(M_E)^{\beta}}\int_{M_E/|z|}^{\infty}s^{n-1} |e(wse^{i\tau})|ds,
\end{align*}
for some $\tilde{c}_2,\,\beta>0$. The estimates in (\ref{e352}) can be applied again in order to arrive at 
\begin{equation}\label{e384}
\left|\int_{M_E/|z|}^{\infty}(se^{i\tau})^n E(zse^{i\tau})\frac{e(wse^{i\tau})}{wse^{i\tau}}e^{i\tau}ds\right|\le C_{1.2}C_{2.2}^{n}m_e(n)\frac{1}{r_1^{n+3}},
\end{equation}
for some $C_{1.2},\,C_{2.2}>0$. From (\ref{e383}) and (\ref{e384}) one can conclude in the spirit of (\ref{e358}) that
\begin{equation}\label{sol1}
\left\|\frac{1}{2\pi i}\int_{\Gamma_1}\varphi(w)\int_0^{\infty(\tau)}\xi^n E(z\xi)\frac{e(w\xi)}{w\xi}d\xi dw\right\|_{\mathbb{E}}\le D_1 D_2^{n}m_e(n)\exp\left(\tilde{M}(D_3|z|)\right),
\end{equation}
for some $D_1,\,D_2,\,D_3>0$.
We continue with the case $w\in\Gamma_2(R)$. The same choice for $\tau$ in (\ref{e379}) holds. We parametrize $\Gamma_2(R)$ by $[r_1,R]\ni\rho\mapsto \rho e^{i(d-\theta_2/2)}$ to arrive at
$$\left|\int_0^{\infty(\tau)}\xi^n E(z\xi)\frac{e(w\xi)}{w\xi}d\xi\right|\le 
\frac{1}{\rho}\int_0^{\infty}s^{n-1}\left|E(zse^{i\tau})\right|\left|e(\rho se^{i\left(\tau+d-\theta_2/2\right)})\right|ds.$$
The same splitting of the path into the segment $[0,M_E/|z|]$ and the ray $[M_E/|z|,\infty)$, and analogous arguments as in the first part of the proof yield
$$\left|\int_{0}^{\infty(\tau)}\xi^n E(z\xi)\frac{e(w\xi)}{w\xi}d\xi\right|\le C_{1.3}C_{2.3}^{n}m_e(n)\frac{1}{\rho^{n+3}},$$
for some $C_{1.3},\,C_{2.3}>0$, and $w=\rho e^{i(d-\theta_2/2)}$ for some $r_1\le \rho\le R$. We derive that 
\begin{multline}
\left\|\int_{\Gamma_2(R)}\varphi(w)\int_0^{\infty(\tau)}\xi^nE(z\xi)\frac{e(w \xi)}{w\xi}d\xi dw\right\|_{\mathbb{E}}\le C_{1.3}C_{2.3}^{n}m_e(n)\frac{1}{2\pi}\int_{r_1}^{R}\left\|\varphi(\rho e^{i(d-\theta_2/2)})\right\|_{\mathbb{E}}\frac{1}{\rho^{n+3}}d\rho\\
\le \frac{c_{\varphi}C_{1.3}}{2\pi r_1^3}\left(\frac{C_{2.3}}{r_1}\right)^{n}m_e(n) R\exp\left(\tilde{M}\left(\frac{R}{\tilde{k}_{\varphi}}\right)\right)
\end{multline} 
for some $c_{\varphi},\,\tilde{k}_{\varphi}>0$ associated with the growth of $\varphi$ near infinity. A direct consequence of the definition of the function $\tilde{M}$, and the definition of the radius $R=R(|z|)$ yield
$$R\exp\left(\tilde{M}\left(\frac{R}{\tilde{k}_{\varphi}}\right)\right)\le \exp\left(\tilde{M}\left(\frac{cR}{\tilde{k}_{\varphi}}\right)\right)=\exp\left(\tilde{M}\left(\frac{c\rho(2) k|z|}{\tilde{k}\tilde{k}_{\varphi}}\right)\right),$$
which allows to end this part of the proof. We get that
\begin{equation}\label{sol2}
\left\|\frac{1}{2\pi i}\int_{\Gamma_2(R)}\varphi(w)\int_0^{\infty(\tau)}\xi^n E(z\xi)\frac{e(w\xi)}{w\xi}d\xi dw\right\|_{\mathbb{E}}\le D_4 D_5^{n}m_e(n)\exp\left(\tilde{M}(D_6|z|)\right),
\end{equation} 
for some $D_4,\,D_5,\,D_6>0$.

The case $w\in\Gamma_4(R)$ can be treated analogously, to arrive at
\begin{equation}\label{sol3}
\left\|\frac{1}{2\pi i}\int_{\Gamma_4(R)}\varphi(w)\int_0^{\infty(\tau)}\xi^n E(z\xi)\frac{e(w\xi)}{w\xi}d\xi dw\right\|_{\mathbb{E}}\le D_7 D_8^{n}m_e(n)\exp\left(\tilde{M}(D_9|z|)\right),
\end{equation} 
for some $D_7,\,D_8,\,D_9>0$.

We conclude the proof with the case that $w\in\Gamma_3(R)$. We parametrize $\Gamma_3(R)$ by $[d-\theta_2/2,d+\theta_2/2]\ni t\mapsto Re^{it}$ and choose $w\in\Gamma_3(R)$. Let $\tau\in \left(-\arg(w)-\frac{\omega(m_e)\pi}{2},-\arg(w)+\frac{\omega(m_e)\pi}{2}\right)$. Then, one has
\begin{align*}
\left|\int_0^{\infty(\tau)}\xi^n E(z\xi)\frac{e(w\xi)}{w\xi}d\xi\right|&=\left|\int_0^{\infty}(se^{i\tau})^n E(zse^{i\tau})\frac{e(wse^{i\tau})}{wse^{i\tau}}e^{i\tau}ds\right|\\
&\le \frac{1}{R}\int_0^{\infty}s^{n-1} |E(zse^{i\tau})||e(wse^{i\tau})|ds.
\end{align*}
In view of (\ref{e191}) and (\ref{e162}), together with the application of the same argument as in (\ref{e329}) (for $r_1$ substituted by $R$),  the previous expression is bounded from above by
\begin{multline}\frac{c\tilde{c}}{R}\int_0^{\infty}s^{n-1}\exp\left(M\left(\frac{|z|s}{\tilde{k}}\right)\right)\exp\left(-M\left(\frac{Rs}{k}\right)\right) ds\\
\le \frac{c\tilde{c}}{R}\int_0^{\infty}s^{n-1}\exp\left(M\left(\frac{s|z|}{\tilde{k}}\right)-2M\left(\frac{sR}{\rho(2)k}\right)\right)ds.
\end{multline}
The function $M$ is monotone increasing in $[0,\infty)$. We recall that $R=\frac{\rho(2)k}{\tilde{k}}|z|$, so the previous expression is bounded from above by
$$\frac{c\tilde{c}}{R}\int_0^{\infty}s^{n-1}\exp\left(-M\left(\frac{sR}{\rho(2)k}\right)\right)ds.$$
At this point, one can take into account (\ref{e352}), together with $R=\frac{\rho(2)k}{\tilde{k}}|z|\ge \frac{\rho(2)k}{\tilde{k}}\tilde{r}$, and
$$\sup_{|w|=R}\left\|\varphi(w)\right\|_{\mathbb{E}}\le c_{\varphi}\exp\left(\tilde{M}\left(\frac{R}{\tilde{k}_{\varphi}}\right)\right)=c_{\varphi}\exp\left(\tilde{M}\left(\frac{\rho(2)k}{\tilde{k}\tilde{k}_{\varphi}}|z|\right)\right)$$
to get that
\begin{equation}\label{sol4}
\left\|\frac{1}{2\pi i}\int_{\Gamma_3(R)}\varphi(w)\int_0^{\infty(\tau)}\xi^n E(z\xi)\frac{e(w\xi)}{w\xi}d\xi dw\right\|_{\mathbb{E}}\le D_{10} D_{11}^{n}m_e(n)\exp\left(\tilde{M}(D_{12}|z|)\right),
\end{equation} 
for some $D_{10},\,D_{11},\,D_{12}>0$.

Statement (\ref{e279}) follows from the application of (\ref{sol1}), (\ref{sol2}), (\ref{sol3}) and (\ref{sol4}). We observe that the identity in (\ref{e244}) is of analytic nature after the deformation path with respect to $w$ and the appropriate choice of $\tau$ described in the proof, for each $z\in \hat{S}_d(\theta_1;\tilde{r})$.
\end{proof}

\begin{corol}\label{coro1}
Let $m_e$ be a sequence of moments, and let $\tilde{\mathbb{M}}$ be a strongly regular sequence admitting a nonzero proximate order. Given $d\in\R$, the space $\mathbb{E}\{z\}_{\tilde{\mathbb{M}},d}$ is closed under $m_e$-differentiation.
\end{corol}
\begin{proof}
Let $\tilde{e},\,\tilde{E}$ be a pair of kernel functions for $\tilde{\mathbb{M}}$-summability, whose existence is guaranteed (see Remark at page~\pageref{r223}). Let $m_{\tilde{e}}$ be its associated sequence of moments. We write $\overline{m}:=(m_e(p)m_{\tilde{e}}(p))_{p\ge0}$, which is a sequence of moments in view of Lemma~\ref{lema268}.
  
Let $\hat{f}\in \mathbb{E}\{z\}_{\tilde{\mathbb{M}},d}$. Then, it holds that $\hat{\mathcal{B}}_{m_{\tilde{e}},z}\hat{f}$ defines a function on some neighborhood of the origin, say $U$, which can be extended to an infinite sector of bisecting direction $d$, say $S_{d}$. Therefore, $\hat{\mathcal{B}}_{m_{\tilde{e}},z}\hat{f}\in\mathcal{O}^{\mathbb{M}}(\hat{S}_d,\mathbb{E})$. We apply the second part of Theorem~\ref{lemma1} to the strongly regular sequence $\overline{m}$ and $n=1$ in order to achieve that $\partial_{\overline{m},z}\hat{\mathcal{B}}_{m_{\tilde{e}},z}\hat{f}$, which is an element in $\mathcal{O}(U)$, is such that $\partial_{\overline{m},z}\hat{\mathcal{B}}_{m_{\tilde{e}},z}\hat{f}\in \mathcal{O}^{\mathbb{M}}(\hat{S}_d,\mathbb{E})$. Lemma~\ref{lema268} yields that 
$$\partial_{\overline{m},z}\hat{\mathcal{B}}_{m_{\tilde{e}},z}\hat{f}\equiv \hat{\mathcal{B}}_{m_{\tilde{e}},z}\partial_{m_{e},z}\hat{f}.$$
We conclude that  $\partial_{m_{e},z}\hat{f}$ defines a holomorphic function on some neighborhood of the origin, and admits an analytic extension to an infinite sector of bisecting direction $d$, with adequate growth at infinity. This entails that $\partial_{m_{e},z}\hat{f}\in\mathbb{E}\{z\}_{\tilde{\mathbb{M}},d}$.
\end{proof}

As a consequence of Corollary~\ref{coro1} and the alternative definition of summable formal power series stated in Theorem~\ref{teo1} the following definition makes sense.

\begin{defin}\label{def487}
Let $\tilde{\mathbb{M}}$ be a strongly regular sequence admitting a nonzero proximate order. Assume that $\hat{u}\in\mathbb{E}\{z\}_{\tilde{\mathbb{M}},d}$, for some $d\in\R$. Let $m_e$ be a sequence of moments. Then, we define the operator of $m_e$-moment differentiation of $\mathcal{S}_{\tilde{M},d}(\hat{u})$ by
$$\partial_{m_e,z}(\mathcal{S}_{\tilde{M},d}(\hat{u})):=\mathcal{S}_{\tilde{\mathbb{M}},d}(\partial_{m_e,z}(\hat{u})).$$
\end{defin}

The previous definition allows to determine upper bounds for the sum of a formal power series in the same way as in Theorem~\ref{lemma1}.

\begin{prop}\label{prop497}
Let $m_{e}=(m_e(p))_{p\ge 0}$ be a sequence of moments. Let $\tilde{\mathbb{M}}$ be a strongly regular sequence admitting a nonzero proximate order and $d\in\R$. We choose $\hat{u}\in\mathbb{E}\{z\}_{\tilde{\mathbb{M}},d}$ and write $u=\mathcal{S}_{\tilde{\mathbb{M}},d}(\hat{u})\in\mathcal{O}(G,\mathbb{E})$ for some sectorial region $G=G_d(\theta)$ with $\theta>\pi\omega(\tilde{\mathbb{M}})$. Then for every $G'\prec G$ there exist $C_4,\,C_5>0$ such that
\begin{equation}\label{e496}
\left\|\partial_{m_e,z}^{n}u(z)\right\|_{\mathbb{E}}\le C_4C_5^nm_e(n)\tilde{M}_n,
\end{equation}
for all $n\in\N_0$ and $z\in G'$.
\end{prop}
\begin{proof}
In view of Theorem~\ref{teo1} one can write $u=\mathcal{S}_{\tilde{\mathbb{M}},d}(\hat{u})=T_{\tilde{e},\tau}(\hat{\mathcal{B}}_{m_{\tilde{e}},z}(\hat{u}))$, for some direction $\tau$ close to $d$, with $\tilde{e}$ being any kernel for $\tilde{\mathbb{M}}$-summability and $m_{\tilde{e}}$ its associated sequence of moments. Taking into account Definition~\ref{def487} and Lemma~\ref{lema268}, one has that 
$$\partial_{m_e,z}^{n}u(z)=T_{\tilde{e},\tau}(\hat{\mathcal{B}}_{m_{\tilde{e}},z}(\partial_{m_e,z}^{n}\hat{u}))(z)=T_{\tilde{e},\tau}(\partial_{m_{\tilde{e}}m_e,z}^{n}(\hat{\mathcal{B}}_{m_{\tilde{e}},z}\hat{u}))(z),$$
for all $n\in\N_0$ and $z\in G_d(\theta)$.

We observe that $\hat{\mathcal{B}}_{m_{\tilde{e}},z}\hat{u}\in\mathcal{O}^{\tilde{\mathbb{M}}}(\hat{S}_d(\delta;r),\mathbb{E})$ for some $\delta>0$ and $r>0$. Therefore, one may apply Theorem~\ref{lemma1} to the sequence of moments $m_em_{\tilde{e}}$ (see Lemma~\ref{lema268}) to arrive at
\begin{equation}\label{e507}
\left\|\partial_{m_em_{\tilde{e}},z}^n(\hat{\mathcal{B}}_{m_{\tilde{e}},z}\hat{u}(z))\right\|_{\mathbb{E}}\le C_1C_2^nm_e(n)m_{\tilde{e}}(n)\exp\left(\tilde{M}(C_3|z|)\right),
\end{equation}
for some $C_1,\,C_2,\,C_3>0$ and $z\in\hat{S}_d(\delta_1;\tilde{r})$ for $0<\delta_1<\delta$, $0<\tilde{r}<r$. Let $f(z):=\partial_{m_em_{\tilde{e}},z}^n(\hat{\mathcal{B}}_{m_{\tilde{e}},z}\hat{u}(z))$. Then, there exists $C_6>0$ such that
$$\left\|\int_0^{\infty(\tau)}\tilde{e}(u/z)f(u)\frac{du}{u}\right\|_{\mathbb{E}}\le\left\|\int_0^{t_0}\tilde{e}(u/z)f(u)\frac{du}{u}\right\|_{\mathbb{E}}+\left\|\int_{t_0}^{\infty(\tau)}\tilde{e}(u/z)f(u)\frac{du}{u}\right\|_{\mathbb{E}}=I_3+I_4\le C_6,$$
for some $\tau\in\hbox{arg}(S_d(\delta_1))$. Usual estimates for $\tilde{e}$-Laplace transform lead to the conclusion: the integrability property of $e$ at the origin (see~(\ref{e192})) leads to upper bounds for $I_3$ and $I_4$ is bounded from above in view of (\ref{e162}) and the very definition of the function $M$. More precisely, this holds for $|\hbox{arg}(z)-\tau|<\omega(\tilde{M})\pi/2$ and small enough $|z|$. One may vary $\tau$ among the arguments $\hbox{arg}(S_d(\delta_1))$ following the usual procedure.

Finally, the bounds in (\ref{e496}) are attained taking into account that $\mathbb{M}$ and $m_{\tilde{e}}$ are equivalent sequences, in view of the remark after Definition~\ref{defi271}.
\end{proof}

\section{Application: Summability of formal solutions of moment Partial Differential Equations}\label{secapp}

This section is devoted to the study of summability properties of the formal solutions of a certain family of moment partial differential equations. 

Let $\mathbb{M}$ be a strongly regular sequence which admits a nonzero proximate order. We assume that $\mathbb{M}_1$ and $\mathbb{M}_2$ are strongly regular sequences which admit nonzero proximate order. Let $e_1$ (resp. $e_2$) be a kernel function for $\mathbb{M}_1$-summability (resp. for $\mathbb{M}_2$-summability), and we write $m_1$ (resp. $m_2$) for its associated sequence of moments. Additionally, we assume that $m_1$ and $m_2$ are $\mathbb{M}$-sequences of orders $s_1>0$ and $s_2>0$, respectively.

Let $1\le k<p$ be integer numbers such that $s_2p>s_1k$. Let $r>0$. We denote $D:=D(0,r)$ and assume that $a(z)\in\mathcal{O}(\overline{D})$ and $a(z)^{-1}\in\mathcal{O}(\overline{D})$. We also fix $\hat{f}\in\C[[t,z]]$ and $\varphi_{j}(z)\in\mathcal{O}(\overline{D})$ for $j=0,\ldots,k-1$.

We consider the following Cauchy problem.
\begin{equation}\label{epral}
\left\{ \begin{array}{lcc}
            \left(\partial_{m_1,t}^{k}-a(z)\partial_{m_2,z}^{p}\right)u(t,z)=\hat{f}(t,z)&\\
             \partial_{m_1,t}^{j}u(0,z)=\varphi_j(z),&\quad j=0,\ldots,k-1.
             \end{array}
\right.
\end{equation}

\begin{lemma}
Under the previous assumptions there exists a unique formal solution $\hat{u}(t,z)\in\C[[t,z]]$ of the Cauchy problem (\ref{epral}). Moreover, in the case that $\hat{f}(t,z)\in\mathcal{O}(\overline{D})[[t]]$ we have $\hat{u}(t,z)\in\mathcal{O}(\overline{D})[[t]]$.
\end{lemma}
\begin{proof}
Let $\hat{u}(t,z)\in\C[[t,z]]$. We write $\hat{u}(t,z)=\sum_{n\ge0}\frac{u_{n,\star}(z)}{m_1(n)}t^{n}$, for some $u_{n,\star}(z)\in\C[[z]]$. The initial conditions of (\ref{epral}) determine $u_{j,\star}(z)=m_1(0)\varphi_j(z)$ in order that $\hat{u}(t,z)$ is a formal solution of (\ref{epral}). We plug the formal power series $\hat{u}(t,z)$ into the problem to arrive at the recurrence formula
\begin{equation}\label{e547}
u_{n+k,\star}(z)=a(z)\partial_{m_2,z}^{p}u_{n,\star}(z)+\hat{f}_{n,\star}(z),
\end{equation}
where we write $\hat{f}(t,z)=\sum_{n\ge0}\frac{\hat{f}_{n,\star}(z)}{m_1(n)}t^{n}$. Therefore, the elements $u_{n,\star}(z)$ for $n\ge k$ are determined by the initial data. Furthermore, under the convergence assumption on $\hat{f}$ the solution of (\ref{e547}) belongs to $\mathcal{O}(\overline{D})$. 
\end{proof}

From now on, the pair $(\mathbb{E},\left\|\cdot\right\|_{\mathbb{E}})$ denotes the Banach space of holomorphic functions in $\overline{D}$, and $\left\|\cdot\right\|_{\mathbb{E}}$ stands for the norm
$$\left\|\sum_{n=0}^{\infty}a_nz^n\right\|_r:=\sum_{n=0}^{\infty}|a_n|r^n.
$$

\begin{lemma}
 \label{le:5}
 Let $m=(m(p))_{p\ge0}$ be a (lc) sequence and $f\in\mathcal{O}(\overline{D})$.
 If there exist $C>0$ and $n\in\mathbb{N}_0$ such that
 \begin{equation}
  \label{eq:f}
  \left\|f(z)\right\|_{\tilde{r}}\leq \frac{|z|^n}{m(n)}C\quad\textrm{for every}\quad z\in\overline{D},\quad \tilde{r}=|z|
 \end{equation}
 then
 \begin{equation*}
  \left\|\partial_{m,z}^{-k}f(z)\right\|_{\tilde{r}}\leq \frac{|z|^{n+k}}{m(n+k)}C\quad\textrm{for every}\quad k\in\mathbb{N}_0
  \quad\textrm{and}\quad z\in\overline{D},\quad \tilde{r}=|z|.
 \end{equation*}
 \end{lemma}
 \begin{proof}
 By (\ref{eq:f}) we may write
  $f(z)=\sum_{j=n}^{\infty}f_jz^j\in\mathcal{O}(\overline{D})$. We define the auxiliary function $g(z)\in\mathcal{O}(\overline{D})$ as
  $$
  g(z):=\sum_{j=0}^{\infty}|f_{j+n}|m(n)z^j.
  $$
  By (\ref{eq:f}) we get $\left\|g(z)\right\|_{\tilde{r}}\le C$
  and $f(z)\ll\frac{z^n}{m(n)}g(z)$. Since $m$ is a (lc) sequence
  we conclude that
  $$
  \partial_{m,z}^{-k}f(z)=\frac{z^{n+k}}{m(n+k)}\sum_{j=0}^{\infty}\frac{m(n+j)m(n+k)}{m(j+n+k)}f_{j+n}z^{j}
  \ll \frac{z^{n+k}}{m(n+k)}\sum_{j=0}^{\infty}|f_{j+n}|m(n)z^j=\frac{z^{n+k}g(z)}{m(n+k)}.
  $$
 Hence
 $$
 \left\|\partial_{m,z}^{-k}f(z)\right\|_{\tilde{r}}\le \frac{|z|^{n+k}}{m(n+k)}\left\|g(z)\right\|_{\tilde{r}}
 \le\frac{|z|^{n+k}}{m(n+k)}C
 $$
 for every  $k\in\mathbb{N}_0$ and $z\in\overline{D}$, $\tilde{r}=|z|$.
 \end{proof}

\begin{theo}\label{teopral}
Under the assumptions made on the elements involved in the Cauchy problem (\ref{epral}) let $\hat{u}(t,z)$ be the formal solution of (\ref{epral}) and $d\in\R$. We define the strongly regular sequence $\overline{\mathbb{M}}=(M_n^{\frac{s_2p}{k}-s_1})_{n\ge0}$. The following statements are equivalent:
\begin{enumerate}
\item[(i)] $\hat{u}(t,z)$ is $\overline{\mathbb{M}}$-summable along direction $d$ as a formal power series in $\mathbb{E}[[t]]$.
\item[(ii)] $\hat{f}(t,z)\in\mathbb{E}[[t]]$ and $\partial_{m_2,z}^{j}\hat{u}(t,0)\in\C[[t]]$ for $j=0,\ldots,p-1$ are $\overline{\mathbb{M}}$-summable along direction $d$. 
\end{enumerate}
If one of the previous equivalent statements holds then the sum of $\hat{u}(t,z)$ is an actual solution of (\ref{epral}).
\end{theo}
\begin{proof}

(i)$\Rightarrow$ (ii). Equation (\ref{e547}) entails that $\hat{f}(t,z)=\sum_{n\ge0}\frac{\hat{f}_{n,\star}(z)}{m_1(n)}t^{n}\in\mathbb{E}[[t]]$ whenever $\hat{u}(t,z)\in\mathbb{E}[[t]]$. In addition to this, the space $\mathbb{E}\{t\}_{\overline{\mathbb{M}},d}$ is a differential algebra, with $\mathcal{S}_{\overline{\mathbb{M}},d}$ respecting the operations of addition, product and derivation (see Proposition 6.20 (i)~\cite{sanzproceedings}) and also under the action of the operator $\partial_{m_2,z}$ (see Corollary~\ref{coro1}). Regarding equation (\ref{epral}) we conclude that $\hat{f}(t,z)\in\mathbb{E}\{t\}_{\overline{\mathbb{M}},d}$.

Let $0\le j\le p-1$. The same argument as before yields that the formal power series $\partial_{m_2,z}^{j}\hat{u}(t,z)\in\mathbb{E}\{t\}_{\overline{\mathbb{M}},d}$. A direct application of the definition of summable formal power series guarantees summability of its evaluation at $z=0$ along direction $d$.

\vspace{0.3cm}
(ii)$\Rightarrow$ (i). Let $\hat{\psi}_0(t):=\hat{u}(t,0)$, and for all $1\le j\le p-1$ let $\hat{\psi}_{j}(t):=\frac{m_2(0)}{m_2(j)}\partial_{m_2,z}^{j}\hat{u}(t,0)$. We also write $\hat{\omega}(t,z):=\partial_{m_2,z}^{p}\hat{u}(t,z)$. We observe that the formal power series $\hat{\omega}(t,z)$ satisfies the equation
$$\left(1-\frac{1}{a(z)}\partial_{m_1,t}^{k}\partial_{m_2,z}^{-p}\right)\hat{\omega}(t,z)=\hat{g}(t,z),$$
with 
$$\hat{g}(t,z):=\frac{1}{a(z)}(\partial_{m_1,t}^k\hat{\psi}_0(t)+z\partial_{m_1,t}^k\hat{\psi}_1(t)+\ldots+z^{p-1}\partial_{m_1,t}^k\hat{\psi}_{p-1}(t)-\hat{f}(t,z)).$$

We write $\hat{\omega}(t,z)$ in the form
$$\hat{\omega}(t,z)=\sum_{q\ge0}\hat{\omega}_q(t,z),$$
where
$$\hat{\omega}_0(t,z):=\hat{g}(t,z),\quad \hbox{ and }\quad \hat{\omega}_q(t,z):=\frac{1}{a(z)}\partial_{m_1,t}^{k}\partial_{m_2,z}^{-p}\hat{\omega}_{q-1}(t,z)\hbox{ for all }q\ge1.$$
Observe that the hypotheses in (ii) together with the properties of differential algebra of $\mathbb{E}\{t\}_{\overline{\mathbb{M}},d}$ guarantee that $\hat{\omega}_0(t,z)\in\mathbb{E}[[t]]$ is $\overline{\mathbb{M}}$-summable in direction $d$. Let $\omega_0(t,z)\in\mathcal{O}(G\times \overline{D})$ denote its sum, where $G$ stands for a sectorial region of opening larger than $\pi\omega(\overline{\mathbb{M}})$ bisected by direction $d$. By Proposition~\ref{prop497}, for all $G'\prec G$, there exist $C_{4},\,C_5>0$ such that
$$\left\|\partial_{m_1,t}^n\omega_0(t,z)\right\|_{r}\le C_4C_5^{n}m_{1}(n)M_n^{\frac{s_2p}{k}-s_1}\le \tilde{C}_1\tilde{C}_2^{n}M_{n}^{\frac{s_2p}{k}},$$
for some $\tilde{C}_1,\,\tilde{C}_2>0$, all $n\in\N_0$ and $t\in G'$. An induction argument allows to state that for every $q\ge0$ the formal power series $\hat{\omega}_{q}(t,z)\in\mathbb{E}[[t]]$ is $\overline{\mathbb{M}}$-summable in direction $d$, with 
$$\left\|\partial_{m_1,t}^n\omega_q(t,z)\right\|_{\tilde{r}}\le \tilde{C}_1C^{q}\tilde{C}_2^{qk+n}M_{qk+n}^{\frac{s_2p}{k}}\frac{|z|^{pq}}{m_2(pq)},$$
for $t\in G'\prec G$, $z\in\overline{D}$ with $\tilde{r}=|z|$ and $C=\left\|\frac{1}{a(z)}\right\|_r$.
Indeed, by Lemma \ref{le:5} and by the inductive hypothesis we get
\begin{multline*}
\left\|\partial_{m_1,t}^n\omega_{q+1}(t,z)\right\|_{\tilde{r}}=\left\|\frac{1}{a(z)}\partial_{m_2,z}^{-p}\partial_{m_1,t}^{k+n}\hat{\omega}_{q}(t,z)\right\|_{\tilde{r}}\le C\frac{|z|^{pq+p}}{m_2(pq+p)}\tilde{C}_1C^{q}\tilde{C}_2^{qk+n+k}M_{qk+n+k}^{\frac{s_2p}{k}}
\end{multline*}
for $t\in G'\prec G$, $z\in\overline{D}$ and $\tilde{r}=|z|$.

We have the following upper bound:
$$\sum_{q\ge0}\left\|\partial_{m_1,t}^n\omega_q(t,z)\right\|_{\tilde{r}}\le 
\tilde{C}_1\tilde{C}_2^n\sum_{q\ge0}(C\tilde{C}_2^k|z|^p)^{q}M_{qk+n}^{\frac{s_2p}{k}}\frac{1}{m_2(pq)}.$$

Due to $(mg)$ condition, the fact that $m_2$ is an $\mathbb{M}$-sequence of order $s_2$ (see also \cite{LMS}, Lemma 8) together with (\ref{e140}) yield
\begin{multline}
M_{qk+n}^{\frac{s_2p}{k}}\frac{1}{m_2(pq)}\le (A_{1}^{qk+n}M_{qk}M_n)^{\frac{s_2p}{k}}\frac{1}{A_3^{pq}M_{pq}^{s_2}}=\frac{A_1^{\frac{s_2p(qk+n)}{k}}}{A_{3}^{pq}}M_{n}^{\frac{s_2p}{k}}\frac{M_{qk}^{\frac{s_2p}{k}}}{M_{pq}^{s_2}}\\
\le \frac{A_1^{\frac{s_2p(qk+n)}{k}}}{A_{3}^{pq}}M_{n}^{\frac{s_2p}{k}}\frac{A_1^{qps_2(k+1)k/2}M_{qp}^{s_2}}{M_{qp}^{s_2}}=\frac{A_1^{\frac{s_2p(qk+n)}{k}}A_1^{qps_2(k+1)k/2}}{A_{3}^{pq}}M_{n}^{\frac{s_2p}{k}},
\end{multline}
for some $A_1,\,A_3>0$
We finally have
$$\sum_{q\ge0}\left\|\partial_{m_1,t}^n\omega_q(t,z)\right\|_{\tilde{r}}\le 
\tilde{C}_1\tilde{C}_4^nM_{n}^{\frac{s_2p}{k}}\sum_{q\ge0}(A_{3}^{-p}A_1^{ps_2+ps_2(k+1)k/2}C\tilde{C}_2^k|z|^p)^{q}.$$
The previous series is convergent for $|z|<\frac{A_{3}}{A_1^{s_2+s_2(k+1)k/2}}\left(\frac{1}{C\tilde{C}_2^k}\right)^{1/p}=:r'$. Therefore, one has that 
$$\omega(t,z):=\sum_{q\ge0}\omega_q(t,z)$$
defines an analytic function on $G\times D(0,r')$. We reduce $r$ in order that $r\le r'$, if necessary, to arrive at 
\begin{equation}\label{e635}
\sum_{q\ge0}\left\|\partial_{m_1,t}^n\omega_q(t,z)\right\|_{\mathbb{E}}\le 
\tilde{C}_3\tilde{C}_4^nM_{n}^{\frac{s_2p}{k}},
\end{equation}
for some $\tilde{C}_3>0$, which is valid for all $t\in G'$.

We show that $\omega(t,z)$ is the $\overline{\mathbb{M}}$-sum of $\hat{\omega}(t,z)=\sum_{q\ge0}\hat{\omega}_q(t,z)\in\mathbb{E}[[t]]$ along direction $d$.

Let $e$ be a kernel function for $\overline{\mathbb{M}}$-summability. Then, for all $q\in\N_0$ it holds that $\omega_q(t,z)=T_{e,d}\hat{\mathcal{B}}_{m_e,t}\hat{\omega}_q(t,z)$ and $\omega(t,z)=\sum_{q\ge0}T_{e,d}\hat{\mathcal{B}}_{m_e,t}\hat{\omega}_q(t,z)$. 

By (\ref{e635}) we get that $T^{-}_{e,d}\omega(t,z)\in\mathcal{O}(D'\times D)$ for some disc at the origin $D'$. Proposition~\ref{prop316} can be applied to arrive at $T^{-}_{e,d}\omega(t,z)\in\mathcal{O}^{\overline{\mathbb{M}}}(S_d,\mathbb{E})$, for some infinite sector $S_d$ with bisecting direction $d$. Hence, $T^{-}_{e,d}\omega(t,z)\in\mathcal{O}^{\overline{\mathbb{M}}}(\hat{S}_d,\mathbb{E})$. 

Finally, convergence of the series and Theorem~\ref{teo324} allow us to write
\begin{multline}
T^{-}_{e,d}\omega(t,z)=T^{-}_{e,d}\sum_{q\ge0}T_{e,d}\hat{\mathcal{B}}_{m_e,t}(\hat{\omega}_q(t,z))=T^{-}_{e,d}T_{e,d}\sum_{q\ge0}\hat{\mathcal{B}}_{m_e,t}(\hat{\omega}_q(t,z))\\
=\sum_{q\ge0}\hat{\mathcal{B}}_{m_e,t}(\hat{\omega}_q(t,z))=\hat{\mathcal{B}}_{m_e,t}\left(\sum_{q\ge0}\hat{\omega}_q(t,z)\right)=\hat{\mathcal{B}}_{m_e,t}\hat{\omega}(t,z).
\end{multline}
Therefore, $\hat{\mathcal{B}}_{m_e,t}\hat{\omega}(t,z)\in\mathcal{O}^{\overline{\mathbb{M}}}(\hat{S}_d\times D)$ and the formal power series $\hat{\omega}(t,z)$ is $\overline{\mathbb{M}}$-summable along direction $d$ (as an element in $\mathbb{E}[[t]]$), with sum given by $\omega(t,z)$.

Assume that one of the equivalent statements holds. Let $f(t,z)$ (resp. $u(t,z)$) be the sum of $\hat{f}(t,z)\in\mathbb{E}[[t]]$ (resp.  $\hat{u}(t,z)\in\mathbb{E}[[t]]$) in direction $d$. Then the function $t\mapsto(\partial_{m_1,t}^{k}-a(z)\partial_{m_2,z}^{p})u(t,z)-f(t,z)$ with values in $\mathbb{E}$ admits null $(\overline{\mathbb{M}})$-asymptotic expansion in a sector of opening larger than $\omega(\overline{\mathbb{M}})\pi$. Watson's lemma (see Corollary 4.12~\cite{sanz}) states that it is the null function, which entails that $u(t,z)$ is an analytic solution of (\ref{epral}), which satisfies the Cauchy data.
\end{proof}

Analogous estimates as in the proof of Theorem~\ref{teopral} can be applied to achieve the next result.

\begin{corol}\label{corofinal}
Assume that $s_1k\ge s_2p$. Under the assumptions made on the elements involved in the Cauchy problem (\ref{epral}) let $\hat{u}(t,z)$ be the formal solution of (\ref{epral}) and $d\in\R$. The following statements are equivalent:
\begin{enumerate}
\item[(i)] $\hat{u}(t,z)$ is convergent in a neighborhood of the origin.
\item[(ii)] $\hat{f}(t,z)$ and $\partial_{m_2,z}^{j}\hat{u}(t,0)$ for $j=0,\ldots,p-1$ are convergent in a neighborhood of the origin. 
\end{enumerate}
\end{corol}

\textbf{Remark:} Theorem~\ref{teopral} is compatible with the results obtained in~\cite{LMS}. Indeed, equation (\ref{epral}) falls into the case $\Gamma=\{(0,p)\}$ in Section 5~\cite{LMS}, and where the associated Newton polygon has one positive slope $k_1$ if and only if $s_2p>s_1k$ and it has no positive slope in the opposite case. Indeed, 
$$\frac{1}{k_1}=\max\left\{0,\frac{s_2p}{k}-s_1\right\}.$$
Theorem 1 in~\cite{LMS} states that the formal solution of the equation $\hat{u}(t,z)=\sum_{n\ge0}u_n(z)t^{n}$ satisfies that for some $0<r'<r$ there exist $C,\,H>0$ such that
$$\sup_{z\in D(0,r')}|u_n(z)|\le CH^n(M_n)^{1/k_1},\quad n\in\N_0.$$
The result remains coherent with Theorem 2,~\cite{LMS}.

 \textbf{Remark:} Theorem~\ref{teopral} is also coherent with the results obtained in~\cite{remy2016} in the Gevrey classical setting. See Theorem 2 and Theorem 3,~\cite{remy2016}.



\end{document}